\documentclass{amsart}

\usepackage{amsmath}
\usepackage{amssymb}
\usepackage{latexsym}
\usepackage{array}
\usepackage{ifthen}
\usepackage{amsthm}

\usepackage[mathscr]{eucal}
\usepackage{enumerate}
\usepackage{mathrsfs}
\usepackage{pgfplots}
\pgfplotsset{compat=1.15}
\usetikzlibrary{arrows}
\usepackage{url}

\address[mbb@uevora.pt]{Manuel B. Branco,
Departamento de Matem\'atica,
Universidade de \'Evora (Portugal)}

\address[ojedamc@unex.es]{Ignacio Ojeda,
Departamento de Matem\'aticas,
Universidad de Extremadura (Spain)}

\address[jrosales@ugr.es]{Jos\'e Carlos Rosales,
Departamento de \'Algebra,
Universidad de Granada (Spain)}

\newtheorem{theorem}{Theorem}[section]
\newtheorem{corollary}[theorem]{Corollary}
\newtheorem{lemma}[theorem]{Lemma}
\newtheorem{proposition}[theorem]{Proposition}

\newtheorem{algorithm}[theorem]{Algorithm}

\theoremstyle{definition}
\newtheorem{definition}[theorem]{Definition}
\newtheorem{example}[theorem]{Example}
\newtheorem{remark}[theorem]{Remark}

\title[The set $\mathscr{L}(m,F)$]
{The set of numerical semigroups of a given
multiplicity and Frobenius number}

\author[{M.B.}~Branco, I.~Ojeda, {J.C.}~Rosales]
{Manuel B. Branco, Ignacio Ojeda,
Jos\'e Carlos Rosales}

\thanks{The research of the first author is supported
by the project FCT PTDC/MAT/73544/2006).
The research of the second author is partially
supported by the research groups FQM-024
(Junta de Extremadura/FEDER funds) and by the
projects MTM2017-84890-P and PGC2018-096446-B-C21
(MINECO/FEDER, UE).
The third author was partially supported by the
research groups FQM-343 FQM-5849
(Junta de Andalucia/Feder) and by the project
MTM2017-84890-P (MINECO/FEDER, UE).}

\begin{document}

\begin{abstract}
We study the structure of the family of numerical
semigroups with fixed multiplicity and Frobenius
number. We give an algorithmic method to compute
all the semigroups in this family. As an application
we compute the set of all numerical semigroups
with given multiplicity and genus.
\end{abstract}

\subjclass[2010]{Primary 20M14; Secondary 11D07}

\keywords{Numerical semigroup,
irreducible numerical semigroup,
multiplicity, Frobenius number, genus}

\maketitle

\medskip
\noindent
\textcopyright{} 2021 Sociedade Portuguesa de Matem\'atica.
Published by EMS Press.

\noindent
This is the author-accepted manuscript of:
M. B. Branco, I. Ojeda and J. C. Rosales,
``The set of numerical semigroups of a given multiplicity
and Frobenius number'',
\emph{Portugaliae Mathematica}
\textbf{78} (2021), no.~2, 147--167.

\noindent
Published version:
\url{https://doi.org/10.4171/PM/2064}

\medskip

%%%%%%%%%%%%%%%%%%%%%%%%%%%%%%%%%%%%%%%%%%%%%%%%%%%%%%%%%%%
%%%%%%%%%%%%%%%%%%%%%%%%%%%%%%%%%%%%%%%%%%%%%%%%%%%%%%%%%%%
\section{Introduction}

Let $\mathbb{N}$ be the set of non-negative integers. A \emph{numerical semigroup} is a submonoid $S$ of $(\mathbb{N},+)$ such that $\mathbb{N}\backslash S$ has finitely many elements; the cardinality of $\mathbb{N} \backslash S$ is the \emph{genus} of $S$, denoted here $\operatorname{g}(S)$. If $S$ is a numerical semigroup, its \emph{multiplicity}, $\operatorname{m}(S)$,  is the smallest positive integer belonging to $S.$ The largest integer not belonging to $S$ is the \emph{Frobenius number} of $S$, denoted $\operatorname{F}(S)$.

There are different methods known for the computation of the set of numerical semigroups with a fixed Frobenius number (see, for example, \cite{BGSP, computer} and \cite{fund-gaps}). These methods admit no \emph{a priori} filtration through multiplicities which makes them very expensive, in computing resources, to compute families of numerical semigroups with small multiplicity and relatively large Frobenius number. 

The main aim of this paper is to provide an algorithmic method that, given two integers $m$ and $F$, computes the  set of all numerical semigroups with Frobenius number $F$ and multiplicity $m$. We denote this set $\mathscr L(m,F)$. 

Following the same strategy as in \cite{computer}, we define an equivalence relation $\sim$ on $\mathscr L(m,F)$ such that each equivalence class contains one and only one irreducible numerical semigroup. Recall that a numerical semigroup is \emph{irreducible} if it cannot be expressed as the intersection of two numerical semigroups properly containing it. We denote by $\mathscr I(m, F)$ the set of all irreducible numerical semigroups with  multiplicity $m$ and Frobenius number $F$. Thus, our equivalence relation establishes a bijection $\mathscr L(m,F)/\!\sim\ \cong \mathscr I(m,F)$ which allows us to compute $\mathscr L(m,F)$ by determining the equivalence class modulo $\sim$ of each semigroup in $\mathscr I(m,F)$. Proceeding in this way, we divide our main objective in two subtasks, namely 1) compute $\mathscr I(m,F)$ and 2) determine the class modulo $\sim$ of an element of $\mathscr I(m,F)$. To this end, we have formulated Algorithms \ref{24} and \ref{14}, respectively. An implementation of these algorithms is available in the GAP \cite{gap} package \texttt{NumericalSgps} \cite{numericalsgps}.

This paper is organized as follows. In Section \ref{S2} we give the necessary and sufficient conditions on the integers $F$ and $m$ for the existence of numerical semigroups with Frobenius number $F$ and multiplicity $m$ (Proposition \ref{1}). In Section \ref{S3}, we define the equivalence relation $\sim$ on $\mathscr L(m,F)$ mentioned above. The main result here is Theorem \ref{6} which states that in each equivalence  class modulo $\sim$ there is one and only one irreducible numerical semigroup. In Section \ref{S5}, we study in depth the structure of $\mathscr I(m,F)$. We show that $\mathscr I(m,F)$ has a structure of rooted tree that allows us to formulate an algorithm for the computation of all the semigroups in $\mathscr I(m,F)$. The root of the tree is the semigroup $C(m,F)$ that we completely determine in Proposition \ref{21}, we emphasize here the important role that plays the \emph{ratio} of a numerical semigroup (see Definition \ref{ratio}). Of course, the main result in this section is Algorithm \ref{24}. Once we have developed an algorithm for the computation of $\mathscr I(m,F)$, in Section \ref{S4}, we perform the computation of the classes modulo $\sim$ which leads to Algorithm \ref{14}.
Finally, in Section \ref{S7}, we show that it is possible to adapt our algorithms for the computation of the set of numerical semigroups with a (suitable) given multiplicity and genus.

Counting the number of numerical semigroups of a given genus (or a given Frobenius Number) is a relevant problem in the theory of Numerical Semigroups (see, \cite{bras} and the survey paper \cite{Delgado}). Algorithm \ref{29}, and Algorithms \ref{24} and \ref{14}, compute the set of numerical semigroups with fixed genus and multiplicity, and  with fixed Frobenius Number and multiplicity, respectively. So by moving the multiplicity appropriately, we can use our algorithms to compute the whole set of numerical semigroups with fixed genus or with fixed Frobenius number. Nevertheless, in our opinion, the strength of our algorithms do not lie in its performance in counting numerical semigroups, one can find more efficient algorithms for this purpose (see, for example, \cite{BGSP, BF1, BF2} and \cite{FH}), what makes the algorithms presented in this paper valuable is that the methods known so far do not allow \emph{a priori} multiplicity filtering, as mentioned earlier.

%%%%%%%%%%%%%%%%%%%%%%%%%%%%%%%%%%%%%%%%%%%%%%%%%%%%%%%%%%%
%%%%%%%%%%%%%%%%%%%%%%%%%%%%%%%%%%%%%%%%%%%%%%%%%%%%%%%%%%%
\section{Preliminaries}\label{S2}

In this section we describe the conditions that $m$ and $F$ must satisfy for the existence of numerical semigroups with multiplicity $m$ and Frobenius number $F$. But first, we need to introduce some notation and recall a couple of well-known results.

Let $\mathcal A$ be a nonempty subset of $\mathbb{N}$. We write $\left\langle \mathcal A\right\rangle$ for the submonoid of $(\mathbb{N},+)$ generated by $\mathcal A$, that is, $$\left\langle \mathcal A\right\rangle :=\{\sum_{i=1}^n \lambda_i\,a_i ~|~ n\in\mathbb{N}\backslash\{0\}, ~ a_1,\ldots , a_n\in \mathcal A,\text{and} ~ \lambda_1,\ldots, \lambda_n\in \mathbb{N}  \}.$$ If $M$ is a submonoid of $(\mathbb{N},+)$ and $\mathcal A$ is a subset of $M$ such that $M=\left\langle \mathcal A\right\rangle$ then we say that $\mathcal A$ is a \emph{system of generators} of $M$. Moreover, if $M\neq \left\langle \mathcal A'\right\rangle$ for all $\mathcal A' \varsubsetneq \mathcal A$, then  we say that $\mathcal A$ is a \emph{minimal system of generators} of $M$.

\begin{lemma}\label{8}\cite[Corollary 2.8]{libro}.
Let $M$ be a submonoid of $(\mathbb{N},+)$. Then $M$ has a unique minimal system of generators, which in addition is finite.
\end{lemma}

Given a submonoid $M$  of $(\mathbb{N},+)$, we denote by  $\operatorname{msg}(M)$ the  minimal system of generators of $M$. 

The following result follows from \cite[Lemma 2.3]{libro}.

\begin{lemma}\label{9}
Let S be a numerical semigroup and $x\in S$.  Then $x\in\operatorname{msg}(S)$ if and only if $S\backslash \{x\}$
is a numerical semigroup.
\end{lemma}

Given $\{x_1\leq x_2\leq \cdots\leq x_n\}\subseteq \mathbb{N}$ we denote by $\{x_1, x_2,\ldots, x_n, \rightarrow\}$ the 
set $\{x_1, x_2,\ldots, x_n\}\cup\{x\in\mathbb{N} ~|~ x > x_n\}$.

Let $m$ and $F$ be two integers and let $\mathscr L(m, F)$ be the set of numerical semigroups with multiplicity $m$ and Frobenius number $F$. It follows from the definitions that if $m \leq 0$ or $F \leq -1$, then $\mathscr L(m, F) = \varnothing$. Moreover, if $m = 1$, then $\mathscr L(m, F) \neq \varnothing $ if and only if $F = -1$. In this case, $\mathscr L(m, F) = \{\mathbb{N}\}$.

\begin{proposition}\label{1}
Let $m$ and $F$ be two integers. If $(m,F) \neq (1,-1)$, then $\mathscr L(m, F)\neq \varnothing$ if only if $F\geq m-1 \geq 1$ and $m\nmid F$.
\end{proposition}

\begin{proof}
If $S \in \mathscr L(m, F)$ then $m-1\not\in S$ and thus $F\geq m-1 \geq 1$.  Furthermore $F\not\in S$ and $\left\langle m\right\rangle\subseteq S$.  Hence $F\not\in \left\langle m\right\rangle$ which implies that $m\nmid F$. Conversely, if $F \geq m-1$ and $m \nmid F,$ then $S = \left\langle m\right\rangle\cup \{F+1, \longrightarrow\}$ is a numerical semigroup of multiplicity $m$ and Frobenius number $F$, that is, $S \in \mathscr L(m, F)$.
\end{proof}

Notice that $\left\langle m\right\rangle\cup \{F+1, \rightarrow\}$ is the (unique) minimal (with respect to the inclusion) element in $\mathscr L(m, F)$, that is to say, $\left\langle m\right\rangle\cup \{F+1, \rightarrow\} \subseteq S$ for every $S \in \mathscr L(m, F)$.

\begin{proposition}\label{3}
Let $m$ be an integer  greater than or equal to two.
\begin{enumerate}[(a)]
\item $\mathscr L(m, m-1)=\big\{\{0,m\rightarrow\}\big\}$.
\item If $m < F < 2m$, then \[\mathscr L(m, F)=\big\{ \{0,m\} \cup A \cup \{F+1,\rightarrow\} ~|~ A \subseteq \{m+1,\ldots,F-1\} \big\}.\]
\item If $\mathscr L(2, F) \neq \varnothing$, then $F$ is odd and furthermore $\mathscr L(2, F)=\{\langle 2,F+2\rangle\} $. 
\end{enumerate}
\end{proposition}

\begin{proof}
(a) If $S$ is a numerical semigroup with multiplicity $m$ and Frobenius number $F = m-1$, then $S \subseteq \{0, m, \rightarrow\}$. Now, since $\left\langle m\right\rangle\cup \{F+1, \rightarrow\} = \{0, m, \rightarrow\},$ our claim follows.

(b) If $m<F<2m$, by \cite[Proposition 6(b)]{IJAC}, one has that the numerical semigroup $\{0,m\rightarrow\}\backslash \{F\}$ is the (unique) maximal (with respect to the inclusion) element of $\mathscr L(m, F)$. Thus, $\left\langle m\right\rangle\cup \{F+1, \rightarrow\} \subseteq S \subseteq \{0,m\rightarrow\}\backslash \{F\}$, for every $S \in  \mathscr L(m, F)$. Therefore, $\mathscr L(m, F) \subseteq \big\{ \{0,m\} \cup A \cup \{F+1,\rightarrow\} ~|~ A \subseteq \{m+1,\ldots,F-1\} \big\} $. Now,  to get the opposite inclusion, it suffices to observe that $S_A = \{0,m\} \cup A \cup \{F+1,\rightarrow\}$ is a numerical semigroup, for every $A \subseteq \{m+1,\ldots,F-1\}$, and that $S_A$ has multiplicity $m$ and Frobenius number $F$.

(c) If $\mathscr L(2, F) \neq \varnothing$, then $F$ cannot be a multiple of $2$, that is, $F$ is odd. Moreover, given $S \in \mathscr L(2, F)$ and an odd integer $x \in S$, it follows that $x + \mathbb{N} \subseteq S$ thus $x \geq F+2$ and we are done.  
\end{proof}

\begin{remark}\label{Depth1}
Let $S$ be a numerical semigroup with multiplicity $m$ and Frobenius number $F$, and write $F+1=qm-r$ for some integers $q$ and $r$ with $0\le r <m$. The integer $q$ is called the \emph{depth} of $S$. Observe that $S$ has depth $2$ if and only $m < F < 2m$. Therefore, Proposition \ref{3}(b) can be understood as a characterization of the numerical semigroups with depth $2$. This characterization was first obtained by Y.~Zhao (see \cite[Proposition 2.1]{Zhao}). 
\end{remark}

%%%%%%%%%%%%%%%%%%%%%%%%%%%%%%%%%%%%%%%%%%%%%%%%%%%%%%%%%%%
%%%%%%%%%%%%%%%%%%%%%%%%%%%%%%%%%%%%%%%%%%%%%%%%%%%%%%%%%%%
\section{A partition of  $\mathscr L(m, F)$}\label{S3}

The goal of this section is to give a partition of $\mathscr L(m,F)$ which will lead to the structure of an algorithmic procedure to compute $\mathscr L(m,F)$. As a consequence of Proposition \ref{3} we concentrate on the case $m\geq 3$ and $F>2m$, that is to say, we will study the families with depth and multiplicity greater than or equal to three. 

Given a numerical semigroup $S$, we will write \[\theta(S)=\{s\in S ~|~m(S)<s< \frac{\operatorname{F}(S)}{2}\}, \] and let $\sim$ be relation on $\mathscr L(m,F)$ defined by \[S\ \sim \ S' ~ \text{if and only if} ~ \theta(S)=\theta(S');\] clearly $\sim$ is an equivalence relation on $\mathscr L(m,F)$. Let $[S]$ denote the class of $S\in \mathscr L(m, F)$ modulo $\sim$, that is, $[S]:=\{ S'\in \mathscr L(m, F) ~|~ S\ \sim\ S'\}$.

\begin{lemma}\label{4}
The class of $S \in \mathscr L(m, F)$ modulo $\sim$ is closed under union and intersection of its elements. That is to say, $[S]$ has a lattice structure given by the inclusion.
\end{lemma}

\begin{proof}
Assume  that $\{S_1,S_2\}\subseteq [S]$ and let us prove that $\{S_1\cap S_2, S_1\cup S_2\}\subseteq [S]$. Clearly, $S_1\cap S_2\in\mathscr L(m, F)$. Moreover, 
 $\theta(S_1)=\theta(S_2)=\theta (S)$ and thus $\theta (S_1\cap S_2)=\theta(S)$. Hence $S_1\cap S_2\in [S]$. Now, we claim that $S_1\cup S_2$ is a semigroup. Indeed, if $x\in S_1\backslash S_2$ and $y\in S_2\backslash S_1$, then $x$ and $y$ are greater than $F/2$, because $\theta(S_1) = \theta(S_2).$ Therefore $x+y > F$ and, consequently, $x+y \in S_1\cup S_2$. Hence we have $S_1\cup S_2\in \mathscr L(m, F)$. Finally, since $\theta( S_1\cup S_2)=\theta( S_1)=\theta( S_2)=\theta(S)$ we conclude that $S_1\cup S_2\in [S]$.
\end{proof}

Observe that \[Z([S])=\bigcap_{S'\in[S]} S'\quad  \text{and}\quad  U([S])=\bigcup_{S'\in[S]} S'\] are the minimum and the  maximum (with respect to the inclusion) of $[S]$, respectively. In particular, both belong to $\mathscr L(m, F)$ by Lemma \ref{4}.

One of the keys to this work is the following result which shows that  $U([S])$ is the unique irreducible numerical semigroup belonging to $[S]$.

\begin{theorem}\label{6}
Let $m$ and $F$ be  positive integers such that  $m\geq 3$, $F> 2m$ and  $m\nmid F$. If $S\in \mathscr L(m,F)$,  then $[S]\cap \mathscr I(m,F)=\{U([S])\}$. Moreover $U([S])= S\cup \{x\in \mathbb{N}\backslash S ~| ~  F-x\not\in S ~\text{and} ~ x>\frac{F}{2}\}$.\end{theorem}

For the sake of completeness, we include the following result which will be used both in the proof of Theorem \ref{6} and in the following sections.

\begin{lemma}\label{5}
Let $S$ be a numerical semigroup.
\begin{enumerate}[(a)]
\item $S$ is irreducible if and only if $S$ is maximal among all numerical semigroups with Frobenius number $\operatorname{F}(S)$.
\item If $h=max\big\{x\in \mathbb{N}\backslash S ~| ~ \operatorname{F}(S)-x\not\in S ~\text{and} ~ x\neq\frac{\operatorname{F}(S)}{2}\big\}$, then $S\cup\{h\}$ is a numerical semigroup with $\operatorname{F}(S\cup\{h\})=\operatorname{F}(S)$ .
\item $S$ is irreducible if and only if $\big\{x\in \mathbb{N}\backslash S ~| ~  \operatorname{F}(S)-x\not\in S ~\text{and} ~ x\neq\frac{\operatorname{F}(S)}{2}\big\}$ is the empty set.
\item $S\cup \big\{x\in \mathbb{N}\backslash S ~| ~  \operatorname{F}(S)-x\not\in S ~\text{and} ~ x>\frac{\operatorname{F}(S)}{2}\big\}$ is an irreducible numerical semigroup.
\end{enumerate}
\end{lemma}

\begin{proof}
It is an immediate consequence of \cite[Lemmas 4 and 5]{computer}.
\end{proof}

\noindent\emph{Proof of Theorem \ref{6}.}
Clearly, since $S \in \mathscr L(m,F)$, we have that $\operatorname{F}(S) = F$. 

Let us see first that $U([S]) \in \mathscr I(m,F)$. If $U([S])$ is not irreducible then, by using Lemma \ref{5}, we have that there exists $h \in\mathbb{N}\backslash S$ such that  $\frac{F}{2}<h<F$ and $U([S])\cup\{h\}\in  \mathscr L(m,F)$. Besides, as $\theta\left(U([S])\cup\{h\}\right)=\theta\left(U([S])\right)=\theta\left(S\right)$, we have that $U([S])\cup\{h\}\in [S]$, in contradiction with the maximality of $U([S])$.  

Now, we prove that if $S'\in [S]\cap \mathscr I(m,F)$ then $S'=U([S])$. In fact, as $S'\in [S]$ then $S'\subseteq U([S])$. Since  both $S'$ and  $U([S])$ belong to $\mathscr I(m,F)$, from Lemma \ref{5}(a), we conclude that  $S'=U([S])$.

Finally, by applying Lemma \ref{5}(d), we get that  $S\cup \big\{x\in \mathbb{N} \backslash S ~| ~  F-x\not\in S ~\text{and} ~ x>\frac{F}{2}\big\}\in  \mathscr I(m,F)\cap [S]$. Therefore $S\cup \big\{x\in \mathbb{N}\backslash S ~| ~  F-x\not\in S ~\text{and} ~ x>\frac{F}{2}\big\}=U([S])$. \qed

\medskip
The following results are immediate consequences of Theorem \ref{6}.

\begin{corollary}
The map $\varphi : \mathscr L(m, F) \to \mathscr I(m,F)$ such that $S \mapsto U([S])$ is surjective and $\varphi(S) = \varphi(S')$ if and only if $\theta(S) = \theta(S').$ In particular, $\mathscr L(m, F)/\! \sim\, \cong \mathscr I(m,F)$.
\end{corollary}

\begin{corollary}\label{7}
Let $m$ and $F$ be  positive integers such that $m\geq 3$, $F>2m$ and  $m\nmid F$. Then $$\mathscr L(m,F)= \bigsqcup_{S\in  \mathscr I(m,F)}[S],$$ where $\sqcup$ means disjoint union.
\end{corollary}

\begin{remark}\label{remAlg}
In view of Corollary \ref{7}, in order to determine explicitly the elements in the set $\mathscr L(m,F)$ we will need
\begin{enumerate}
\item an algorithm to compute the set $\mathscr I(m,F)$
\item an algorithm to compute the set $[S]$, for each  $S\in \mathscr I(m,F)$.
\end{enumerate}
These algorithms will be developed in Sections \ref{S5} and \ref{S4}, respectively.
\end{remark}

%%%%%%%%%%%%%%%%%%%%%%%%%%%%%%%%%%%%%%%%%%%%%%%%%%%%%%%%%%%
%%%%%%%%%%%%%%%%%%%%%%%%%%%%%%%%%%%%%%%%%%%%%%%%%%%%%%%%%%%
\section{An algorithm for the computation of $\mathscr I(m, F)$}\label{S5}

Our main goal in this section is to describe all the elements in $\mathscr I(m, F)$. To do that, we first describe the conditions that $m$ and $F$ must verify such that there exists at least one irreducible numerical semigroup with multiplicity $m$ and Frobenius number $F$.

\begin{lemma}\cite[Lemma 8]{IJAC}.\label{15}
Let $S$ be a numerical semigroup. Then $S$ is irreducible if and only if $\operatorname{g}(S) = \Big\lceil \frac{\operatorname{F}(S)+1}{2} \Big\rceil$.
\end{lemma}

\begin{proposition}\label{19}
Let $m$ and $F$ be  positive integers such that $F\geq 3$. Then $\mathscr I(m,F)\neq\varnothing$ if and only if $m\leq \frac{F+2}{2}$ and $m\nmid F$.
\end{proposition} 

\begin{proof}
Let $S \in \mathscr I(m,F)$, in particular, $\operatorname{F}(S) = F$. If $m|F$ then we have $F\in\langle m\rangle\subseteq S$, which is impossible. As $\{1,\ldots,m-1\}\cup\{F\}\in \mathbb{N}\backslash S$ it follows that $m\leq \operatorname{g}(S)$, from Lemma \ref{15} we get  that $m\leq \frac{F+2}{2}$. Conversely, it is clear that $S=\langle m\rangle\cup \{F+1,\rightarrow\}$ belongs to $\mathscr L(m,F)$. By using Lemma \ref{5}(d) we obtain that  $\overline{S}= S\cup \{x\in \mathbb{N}\backslash S ~| ~  F-x\not\in S ~\text{and} ~ x>\frac{F}{2}\}$ is in
$\mathscr I(m,F)$.
\end{proof}

It is well known (see for instance \cite {pacific}) that the class of irreducible numerical semigroups  is the disjoint union of two sub-classes of particular interest, which are called: \emph{symmetric} and \emph{pseudo-symmetric numerical semigroups} (see \cite {Kunz, froberg, barucci}). There are several characterizations for these classes of numerical semigroups. The next result is one of many and it will be used extensively in what follows.

\begin{lemma}\cite[Proposition 4.4]{libro}.\label{20}
Let $S$ be a numerical semigroup. 
\begin{enumerate}[(a)]
\item  $S$ is symmetric if and only if $\operatorname{F}(S)$  is odd and $x\in \mathbb{N}\backslash S$ implies $\operatorname{F}(S)-x\in S$.
\item  $S$ is pseudo-symmetric if and only if $\operatorname{F}(S)$ is even and $x\in \mathbb{N}\backslash S$ implies that
either $\operatorname{F}(S)-x\in S$ or $x =\frac{\operatorname{F}(S)}{2}$.
\end{enumerate}
\end{lemma} 

Note that a numerical semigroup  is symmetric (respectively pseudo-symmetric)  if it is irreducible with odd (respectively even) Frobenius number.

\begin{definition}\label{ratio}
Given a numerical semigroup $S$, the smallest element in $S\left\backslash\langle \operatorname{m}(S)\right\rangle$ is called the \emph{ratio} of $S$ and will be denoted by $\operatorname{r}(S)$.
\end{definition}

Notice that if $S$ is a numerical semigroup and $\operatorname{msg}(S)=\{n_1<n_2<\cdots <n_p\}$, then we have $n_1=\operatorname{m}(S)$ and $\operatorname{r}(S)=n_2$.

\begin{proposition}\label{21}
Let $m$ and $F$ be  positive integers. If $F\geq 3,\  m\leq \frac{F+2}{2}$ and  $m\nmid F$, then there exists a unique irreducible numerical semigroup, denoted by $C(m,F)$, with Frobenius number $F$, multiplicity $m$ and ratio greater than $\frac{F}{2}$. More precisely, 
\begin{enumerate}[(a)]
\item  If $F$ is odd and $1 \leq k \leq m$ is the smallest integer such that 
$\frac{F+1}{2}$ is congruent to $k$ modulo $m$, then  $C(m,F)$ is the numerical semigroup with minimal system of generators $\{m\}\cup \{\frac{F+1}{2}+x ~|~ x\in\{0,\ldots,m-1\}\backslash\{m-k,k-1\}\}$.
\item  If $F$ is even and $1 \leq k \leq m$ the smallest integer such that 
$\frac{F+2}{2}$ is congruent to $k$ modulo $m$, then $C(m,F)$ is the numerical semigroup with minimal system of generators $\{3,F/2+3,F+3\}$, if $m=3$, or $\{m\}\cup \{\frac{F+2}{2}+x ~|~ x\in\{0,\ldots,m-1\}\backslash\{m-k, k-2\}\}$, if $m \neq 3$.
\end{enumerate}
\end{proposition} 

\begin{proof}
Let $S$ be an irreducible numerical semigroup such that $\operatorname{F}(S)=F$, $\operatorname{m}(S)=m$ and $\operatorname{r}(S)>\frac{F}{2}$. We distinguish two cases:

(a) If $F$ is odd, it is clear that $\{\frac{F+1}{2}-1,\ldots,\frac{F+1}{2}-m\}\backslash\{\frac{F+1}{2}-k\}\subseteq \mathbb{N}\backslash S$ and by applying (a) in  Lemma \ref {20} we deduce that $\{\frac{F+1}{2},\ldots,\frac{F+1}{2}+m-1\}\backslash\{\frac{F+1}{2}+k-1\}\subseteq S.$ Consequently, $S$ contains $C(m,F)$. In order to prove the equality, it suffices to show that both have the same genus and, by Lemma \ref{15}, it is enough to prove that $C(m,F)$ is irreducible. Since $F$ is odd, this is the same as $C(m,F)$ being a symmetric numerical semigroup. If $x>\frac{F}{2}$ and $x\not\in C(m,F)$ and considering the system of generators of $C(m,F)$ we deduce that $x=\frac{F+1}{2}+k-1+l m$ for some $l\in\mathbb{N}$ and thus $F-x=\frac{F+1}{2}-k-l m$. Since $\frac{F+1}{2}$ is $k$ modulo $m$, we obtain $F-x\in \langle m\rangle\subseteq C(m,F)$. From Lemma \ref{20}, we can guarantee that $C(m,F)$ is a symmetric numerical semigroup.

(b)  If $F$ is odd, then $C(3,F)$ is the numerical semigroup with minimal system of generators $\{3,F/2+3,F+3\}.$ So, we assume $m > 3$. Then $\{\frac{F}{2}-1,\ldots,\frac{F}{2}-m\}\backslash\{\frac{F}{2}-(r-1)\}\subseteq \mathbb{N}\backslash S$. By using (b) in  Lemma \ref{20} we get that  $\{\frac{F+2}{2},\ldots,\frac{F+2}{2}+m-1\}\backslash\{\frac{F+2}{2}+k-2\}\subseteq S$ and thus  $S$ contains $C(m,F)$. To conclude the proof, we need to prove that $C(m,F)$ is irreducible, as in the previous case. To see this, it suffices to check that  $C(m,F)$ is a pseudo-symmetric numerical semigroup. If $x>\frac{F}{2}$ and  $x\not\in C(m,F)$ we get that $x=\frac{F}{2}+k-1+lm$ for some $l\in\mathbb{N}$ and so $F-x=\frac{F+2}{2}-k- l m$ is a multiple of $m$. Consequently, $F-x\in\langle m\rangle\subseteq C(m,F)$. By applying Lemma \ref{20}(b), we have that $C(m,F)$ is a pseudo-symmetric numerical semigroup.
\end{proof}

The following result is \cite[Proposition 2.5]{forum}.

\begin{proposition}\label{17}
Let $S$ be an irreducible numerical. If $x\in \operatorname{msg}(S)$ verifies $x<\operatorname{F}(S)$, $2x-\operatorname{F}(S)\not\in S$,  $3x\neq 2 \operatorname{F}(S)$ and  $4x\neq 3\operatorname{F}(S)$, then $(S\backslash \{x\})\cup \{\operatorname{F}(S) - x\}$ is an irreducible numerical semigroup with Frobenius number $\operatorname{F}(S)$.
\end{proposition}

\begin{corollary}\label{18}
If $S$ is an irreducible numerical semigroup such that  $\operatorname{r}(S)< \frac{\operatorname{F}(S)}{2}$, then $$\overline{S}=\left(S\backslash \{\operatorname{r}(S)\}\right)\cup \{\operatorname{F}(S) - \operatorname{r}(S)\}$$ is an irreducible numerical semigroup with $\operatorname{F}(\overline{S})=\operatorname{F}(S)$ and $\operatorname{r}(\overline{S})>\operatorname{r}(S)$.
\end{corollary}

\begin{proof}
By hypothesis, $\operatorname{r}(S) < 2\operatorname{r}(S) < \operatorname{F}(S)$, $3\operatorname{r}(S)\neq 2 F$ and  $4\operatorname{r}(S)\neq 3 \operatorname{F}(S)$. In particular, $2\operatorname{r}(S)-\operatorname{F}(S)<0$. So $2\operatorname{r}(S)- \operatorname{F}(S)\not\in S$ and, by Proposition \ref{17}, we can conclude that $\overline{S}$ is an irreducible numerical semigroup with $\operatorname{F}(\overline{S})=\operatorname{F}(S)$. Furthermore, we have  $\operatorname{r}(\overline{S})>\operatorname{r}(S)$  due to $\operatorname{F}(S)-\operatorname{r}(S)>\operatorname{r}(S)$.
\end{proof}

Consider now the function \begin{equation}\label{ecu_psi} \psi : \mathscr I(m,F) \to \mathscr I(m,F)\end{equation} such that $\psi(S) = \left(T \backslash \{\operatorname{r}(T)\}\right)\cup  \{F-\operatorname{r}(S)\}$ if $\operatorname{r}(S) < F/2$, and $\psi(S) = S$ otherwise.

\begin{proposition}\label{28b}
For each $S\in \mathscr I(m,F)$ there exists $\{S_0, \ldots, S_{n}\} \subseteq \mathscr I(m,F)$, such that $S_0 = S,\ S_n = C(m,F)$ and $\psi(S_{i-1}) = S_{i},\ i = 1, \ldots, n$.
\end{proposition}

\begin{proof}
If $\operatorname{r}(S) > F/2,$ then $S = S_0 = C(m,F)$ by Proposition \ref{21}. Otherwise, by Corollary \ref{18}, there exists $S_1 \in \mathscr I(m,F)$ such that $S_1 = \psi(S = S_0)$ and $\operatorname{r}(S_1) > \operatorname{r}(S)$. By repeating this argument with $S_1$, we will obtain that either $S_1 = C(m,F)$ or that there exists $S_2$ such that $S_2 = \psi(S_1)$ and $\operatorname{r}(S_2) > \operatorname{r}(S_1)$. Since this process cannot continue indefinitely, our claim follows.
\end{proof}

\begin{example}\label{22}
Let $S=\langle 6,8,9\rangle=\{0,6,8,9,12,14,15,16,17,18,20,\rightarrow\}$. It is an irreducible numerical semigroup with $\operatorname{m}(S)=6$ and $\operatorname{F}(S)=19$.
\begin{itemize}
\item $S_1=\left(S_0\backslash\{8\}\right)\cup\{11\}=\langle 6,9,11,14,16\rangle$,
\item $S_2=\left(S_1\backslash\{9\}\right)\cup\{10\}=\langle 6,10,11,14,15 \rangle=C(6,19)$.
\end{itemize}
Clearly, $\psi(S) = S_1$ and $\psi(S_1) = S_2 = C(6,19)$.
\end{example}

\begin{theorem}\label{23}
Let $m$ and $F$ be  positive integers. If $F\geq 3$,  $m\leq \frac{F+2}{2}$ and $m\nmid F$, then the semigroups $T \in \mathscr I(m,F)$ such that $\psi(T) = S$ is equal to $\{\big(S\backslash \{x\}\big)\cup \{F-x\}~| ~ x\in \alpha(S)\}$ where
 $$\alpha(S) := \left\{x\in\operatorname{msg}(S) \left| \begin{array}{ccc} \frac{F}{2}<x<F,\ 2x-F\not\in S\\ ~ 3x\neq 2F,\ 4x\neq 3F\\ \operatorname{m}(S)<F-x< \operatorname{r}(S) \end{array} \right. \right\}.$$
\end{theorem}

\begin{proof}
Let $S\in \mathscr I(m,F)$ and $x\in\operatorname{msg}(S)$ such that  $\frac{F}{2}<x<F$,  $2x-F\not\in S$,  $3x\neq 2F$, $4x\neq 3F$ and $\operatorname{m}(S)<F-x< \operatorname{r}(S)$. By Proposition \ref{17} we know that $T=\left(S\backslash \{x\}\right)\cup  \{F-x\}$ is an irreducible numerical semigroup with Frobenius number $F$. Furthermore,  since $\operatorname{m}(S)<F-x< \operatorname{r}(S)$   we obtain $\operatorname{m}(T)=m$ and $\operatorname{r}(T)=F-x < F/2$. Hence, we have 
$S=\left(T\backslash \{\operatorname{r}(T\}\right)\cup  \{F-\operatorname{r}(T)\}$ and so $\psi(T) = S$.

Now, let $S$ and $T \in \mathscr I(m,F)$ be such that $\psi(T) = S$. Since, $\operatorname{r}(T)<\frac{F}{2}$ and $S=\left(T\backslash \{\operatorname{r}(T)\}\right)\cup  \{F-\operatorname{r}(T)\}$, then $T=\left(S\backslash \{F-\operatorname{r}(T)\}\right)\cup  \{F-\left(F-\operatorname{r}(T)\right)\}$. To conclude the proof it is enough to see that $F-\operatorname{r}(T)\in \operatorname{msg}(S)$ verifies the conditions given in the definition of $\alpha(S)$. Clearly $F-\operatorname{r}(T)\in \operatorname{msg}(S)$ because $F-\operatorname{r}(T)\not\in T$. Since $\operatorname{r}(T)<\frac{F}{2}$ we obtain $\frac{F}{2}<F-\operatorname{r}(T)<F$. Moreover, as $2\operatorname{r}(T)\in T\backslash \{\operatorname{r}(T)\}$ then $2\operatorname{r}(T)\in S$ and thus $F-2\operatorname{r}(T)\not\in S$. Consequently, $2\left(F-\operatorname{r}(T)\right)\not\in S$. If $3\left(F-\operatorname{r}(T)\right)=2F$ we would obtain that $F=3\operatorname{r}(T)\in S$, which is impossible.  Furthermore, if $4\left(F-\operatorname{r}(T)\right)=3F$ we would get that $F=4\operatorname{r}(T)\in S$, which is impossible.  Finally, since $S=\left(T\backslash \{\operatorname{r}(T)\}\right)\cup  \{F-\operatorname{r}(T)\}$ and $F-\operatorname{r}(T)>\operatorname{r}(T)$ then $\operatorname{m}(S)<\operatorname{r}(T)<\operatorname{r}(S)$.
\end{proof} 

Now, we are ready to formulate our algorithm that will allow us to compute the set $\mathscr I(m,F)$.

\begin{algorithm}\label{24}\mbox{}\par
{Input:}  $m$ and $F$, positive integers with $F\geq 3$,  $m\leq \frac{F+2}{2}$ and $m\nmid F$.\par
{Output:} The set $\mathscr I(m,F)$.
\begin{itemize}
\item[1.]  Set $A:=\{C(m,F) \}$ and $\mathscr I(m,F) := A$.
\item[2.]  While $A \neq \varnothing$ do
\begin{itemize}
\item[2.1.]  For $S\in A$ 
\begin{itemize}
\item[2.1.1.] Compute $\alpha(S)$.
\item[2.1.2.] Set $E(S):=\{\left(S\backslash \{x\}\right)\cup  \{F-x\} ~|~x\in\alpha (S)\}$.
\end{itemize}
\item[2.2.] Set $A:= \{E(S) \mid S \in A\}$ and $\mathscr I(m,F) := \mathscr I(m,F) \cup A$.
\end{itemize}
\item[3.] Return $\mathscr I(m,F)$.
\end{itemize}
\end{algorithm}

This algorithm computes the set $\mathscr I(m,F)$ starting from $C(m,F)$ whose existence is proved in Proposition \ref{21}. The correctness of the algorithm relies on Theorem \ref{23}.

\begin{remark}
The GAP \cite{gap} package \texttt{NumericalSgps} \cite{numericalsgps} includes the function {\small \texttt{IrreducibleNumericalSemigroupsWithFrobeniusNumberAndMultiplicity}} that compu\-tes the set of numerical semigroups with Frobenius number $F$ and multiplicity $m$ using Algorithm \ref{24}.

Observe that, since we do not need to compute the whole set of irreducible numerical semigroups with Frobenius number $F$ and then restrict our search to those of multiplicity $m$, our algorithm is a real improvement in computing resources means.
\end{remark}

In order to complete the study of $\mathscr I(m,F)$, we end this section by showing the tree structure of $\mathscr I(m,F)$, which is somehow the structure of Algorithm \ref{24}, and by giving an interpretation of $\mathscr I(m,F)$ as the set of solutions of certain integer programs which makes Algorithm \ref{24} an algorithm to solve them.

%%%%%%%%%%%%%%%%%%%%%%%%%%%%%%%%%%%%%%%%%%%%%%%%%%%%%%%%%%%
\subsection*{The tree of $\mathscr I(m,F)$}

Let $m$ and $F$ positive integers such that $F \geq 3,\ m \leq \frac{F+2}2$ and $m \nmid F$. 

Let $G(\mathscr I(m,F))$ be the graph whose vertex set is equal to $\mathscr I(m,F)$ and such that $(T,S) \in \mathscr I(m,F) \times \mathscr I(m,F)$ is an edge if and only if $\psi(T) = S$.

\begin{corollary}
The graph $G(\mathscr I(m,F))$ is a tree with root equal to $C(m,F)$.  
\end{corollary}

\begin{proof}
Let $S \in \mathscr I(m,F)$. By Proposition \ref{28b}, there exists a path from $S$ to $C(m,F)$. If there exists another path from $S$ to $C(m,F)$, then there are $S'$ and $S''$ such that $\psi(S) = S'$ and $\psi(S) = S''$ but this is not possible by the definition of $\psi$.
\end{proof}

\begin{corollary}\label{Cor 25}
If $(S,T)$ is an edge of $G(\mathscr I(m,F))$, then $S \not\subset T$ and $T \not\subset S$.
\end{corollary}

\begin{proof}
Since $S = \psi(T)$, by definition, $S = \big(T\backslash \{\operatorname{r}(T)\}\big)\cup  \big\{F-\operatorname{r}(T)\big\}$, then $\operatorname{r}(T) \in T \backslash S$ and $F-\operatorname{r}(T) \in S \backslash T$.
\end{proof}

Notice that this last corollary says that no edge of $G(\mathscr I(m,F))$ is an edge of the tree introduced in \cite{tree}.

\begin{example}
Figure \ref{Fig1} shows the graph of the irreducible numerical semigroups with 
multiplicity $6$ and Frobenius number $23$. The vertices are labeled by the minimal generators of the corresponding numerical semigroup; for instance, $\{6,13,14,15, 16\}$ is the minimal generators of the root, $C(6,23)$.

\begin{figure}[h]
\begin{center}

\begin{tikzpicture}[line cap=round,line join=round,>=triangle 45,x=1cm,y=1cm]
\clip(-1.25,-0.5) rectangle (7.5,4.5);
\draw [->,line width=1pt] (4,4) -- (1,2);
\draw [->,line width=1pt] (4,4) -- (3,2);
\draw [->,line width=1pt] (4,4) -- (5,2);
\draw [->,line width=1pt] (4,4) -- (7,2);
\draw [->,line width=1pt] (1,2) -- (0,0);
\draw [->,line width=1pt] (1,2) -- (2,0);
\begin{scriptsize}
\draw (4,4.25) node {$\{ 6, 13, 14, 15, 16 \}$};
\draw (-0.25,2) node {$\{ 6, 10, 14, 15, 19 \}$};
\draw (3,1.75) node {$\{ 6, 9, 13, 16, 20 \}$};
\draw (5,1.75) node {$\{ 6, 8, 13 \}$};
\draw (7,1.75) node {$\{ 6, 7, 15 \}$};
\draw (0,-0.25) node {$\{ 6, 9, 10 \}$};
\draw (2,-0.25) node {$\{ 6, 8, 10, 19, 21 \}$};
\end{scriptsize}
\end{tikzpicture}
\caption{The graph $G(\mathscr I(6,23))$.}\label{Fig1}
\end{center}

\end{figure}

\end{example}

%%%%%%%%%%%%%%%%%%%%%%%%%%%%%%%%%%%%%%%%%%%%%%%%%%%%%%%%%%%
\subsection*{On the Kunz coordinates of $\mathscr I(m,F)$}

Finally in this section we would like to draw attention to the conditions that the Kunz coordinates of the semigroups in $\mathscr I(m,F) $ must satisfy. This will provide an interpretation of $\mathscr I(m,F)$ as the solution set of a particular family of integer programs (see \eqref{IP}) which turns the algorithm \ref{24} into an \emph{solver} for them. For self-containedness of the paper, we first recall what Kunz coordinates are.

In the following, $S$ will denote a numerical semigroup of multiplicity $m$ and Frobenius number $F$.

\begin{definition}
The \emph{Ap\'ery set} of $S$ with respect to $m$ is the set $$\mathrm{Ap}(S,m) := \{ s \in S ~|~ s-m \not\in S\}.$$
\end{definition} 

It is known that $\{\mathrm{Ap}(S,m)\backslash\{0\}\big\} \cup \{m\}$ is a (non-necessarily minimal) system of generators of $S$ (see \cite[Lema 1.6]{libro}). Moreover one has that $\mathrm{Ap}(S,m) = \{0 = w(0), w(1), \ldots, w(m-1)\}$, where $w(i), i \in \{0, \ldots, m-1\}$, is the least element in $S$ whose remainder under division by $m$ is $i$. Therefore, $(q_1, \ldots, q_{m-1}) \in \left(\mathbb{N}\backslash \{0\}\right)^{m-1}$ characterizes $\mathrm{Ap}(S,m)$ and vice versa. The vector $(q_1, \ldots, q_{m-1})$ is called the \emph{Kunz coordinates vector} of $S$.

An easy computation shows that
\begin{align*}
w(i) + w(j) & = (q_i+q_j)m + i+j \geq \\ 
& \geq \left\{
\begin{array}{lcl}
q_{i+j} m + i+j &\text{if}& i+j \leq m-1\\
q_{i+j-m} m + i+j-m &\text{if}& i+j > m
\end{array}
\right.\end{align*}
for every $i,j \in \{0, \ldots, m-1\}$.
Thus, we can see $(q_1, \ldots, q_{m-1}) \in \left(\mathbb{N}\backslash \{0\}\right)^{m-1}$ as a solution of the following system of inequalities:
$$  
\begin{array}{ccl} 
x_i \geq 1 & & 1 \leq i \leq m-1,\\
x_i + x_j - x_{i+j-m \delta} \geq -\delta & & 1 \leq i \leq j \leq m-1,
\end{array}
$$
where $\delta = \lfloor \frac{i+j}m \rfloor$ being $\lfloor - \rfloor$ the floor operator. 

Moreover, since $F = \max\{w(i)-m \mid i = 1, \ldots, m -1\}$ (see \cite[Proposition 2.12]{libro}), a set of extra constraints must be considered. 

Thus, putting all this together we obtain the convex lattice polytope defined by:
\begin{equation}\label{F}
\begin{array}{ccl}
x_i = \frac{F-i}m+1,& & \text{if}\ i \equiv F\ \mod m\\
1 \leq x_i < \frac{F-i}m+1,& & i = 1, \ldots, m-1,\ i \not\equiv F\ \mod m,\\
x_i + x_j - x_{i+j-m \delta} \geq -\delta & & 1 \leq i \leq j \leq m-1,\ i+j \neq m\\
 x_i \in \mathbb{N},& & i = 1, \ldots, m-1,
\end{array}
\end{equation}
where $\delta = \lfloor \frac{i+j}m \rfloor$.

The following result is a direct consequence of Proposition \ref{1}.

\begin{corollary}
If $F \geq m-1 \geq 1$ and $m \nmid F$, then there exists, at least, one point with integer coordinates satisfying \eqref{F}.
\end{corollary}

The following result establishes the relationship between the Apery sets of a semigroup in $\mathscr I(m,F)$ and its image by the function $\psi$ defined in \eqref{ecu_psi}.

\begin{proposition}\label{37}
Let $T \in \mathscr I(m,F)$. If $S = \psi(T)$, then 
$$\mathrm{Ap}(S,m) = \Big(\mathrm{Ap}(T,m)\backslash \{\operatorname{r}(T),m+F-\operatorname{r}(T)\}\Big) \bigcup \{F-\operatorname{r}(T), m+\operatorname{r}(T)\}.$$
\end{proposition}

\begin{proof}
By definition, $r(T) < F/2$ and $S = (T\backslash\{\operatorname{r}(T)\}) \cup \{F-\operatorname{r}(T)\}$. Since $\operatorname{r}(T)$ is a minimal generator, it belongs to $\mathrm{Ap}(T,m)$, and $\operatorname{r}(T) \not\in S$.
But $m + \operatorname{r}(T) \in S$ and $m + \operatorname{r}(T) \equiv \operatorname{r}(T)\ \mod m$, and thus $m + \operatorname{r}(T) \in \mathrm{Ap}(S,m)$. On the other hand, $\operatorname{r}(T) - m \not\in T$, then, by Lemma \ref{20}, $F-\operatorname{r}(T) + m \in T$ and $F-\operatorname{r}(T)+m \in \mathrm{Ap}(T,m)$. Now, since $\operatorname{r}(T) \not\in S$, then $F - \operatorname{r}(T) \in S$ and, by Lemma \ref{20} again, $F - \operatorname{r}(T) \equiv F-\operatorname{r}(T)+m\ \mod m$, we conclude that $F - \operatorname{r}(T) \in \mathrm{Ap}(S,m)$.
\end{proof}

Now, we exhibit the equations of the convex sets in $\mathbb{R}^{m-1}$ whose points with integer coordinates are the Kunz coordinates of the irreducible numerical semigroups.

The next result is well known and is easy to prove.

\begin{lemma}\label{38}  
Let $S$ be a numerical semigroup, $m\in S\backslash\{0\}$ and  $\mathrm{Ap}(S,m) = \{0, q_1m+1, \ldots, q_{m-1}m+m-1\}$,  then 
$\operatorname{g}(S)=q_1+\cdots + q_{m-1}$.
\end{lemma}

From Lemmas \ref{15} and \ref{38}, we obtain the following result.

\begin{proposition}\label{39}
If $F \geq 3, m \leq \frac{F+2}2$ and $m \nmid F$, then the sum of the Kunz coordinates vector of $S \in \mathscr I(m,F)$   is equal to  $\lceil \frac{F+1}{2} \rceil$.
\end{proposition}

Therefore, the Kunz coordinate vector of the semigroup $S \in \mathscr I(m,F)$ lies in the hyperplane $x_1 +\cdots + x_{m-1} =  \lceil \frac{F+1}{2} \rceil$, for some positive integer  $\lceil \frac{F+1}{2} \rceil$ that only depends on $m$ and $F$. 

The next result is a consequence of \cite[Propositions 4.10 and 4.15]{libro}. 
\begin{proposition}\label{40}
Let $q$ and $r$ be the quotient and the remainder of the division of $\lceil \frac{F+1}{2} \rceil$  by $m$  and let 
$\left(q_1,\ldots, q_{m-1}\right)$ be the Kunz coordinates vector of  $S \in \mathscr I(m,F)$. Then 
$q_i+q_j + \delta_1 = 2q +1+ \delta_2$ for every $1 \leq i \leq j \leq m-1$ such that $i+j \equiv F\ \mod m$.
\end{proposition}

In conclusion, the Kunz coordinates of $S \in \mathscr I(m,F)$ must verify the following:
\begin{equation}\label{IP}
\begin{array}{ccl}
x_1 + \ldots + x_{m-1} = \lceil \frac{F+1}2\rceil\\
x_i = \frac{F-i}m+1,& & \text{if}\ i \equiv F\ \mod m \\
1 \leq x_i < \frac{F-i}m+1,& & i = 1, \ldots, m-1,\ i \not\equiv F\ \mod m,\\
x_i + x_j - x_{i+j-m \delta_1} \geq -\delta_1 & & 1 \leq i \leq j \leq m-1,\ i+j \neq m\\
x_i + x_j + \delta_1 = 2q + 1+ \delta_2 & & 1 \leq i \leq j \leq m-1,\ i+j \equiv F\ \mod m \\
 x_i \in \mathbb{N},& & i = 1, \ldots, m-1,\\
\end{array}
\end{equation}
where $\delta_1 = \lfloor \frac{i+j}m \rfloor$ and $\delta_2 = \lfloor \frac{2r-1}m \rfloor$.

Note that  as a consequence of  \cite[Propositions 4.10 and 4.15]{libro} we have that the integer solutions of the previous system are the Kunz coordinates of an irreducible numerical semigroup with Frobenius number $F$ and mul-
tiplicity $m$. Therefore, our Algorithm \ref{24} can be seen as a solver of \eqref{IP}.

%%%%%%%%%%%%%%%%%%%%%%%%%%%%%%%%%%%%%%%%%%%%%%%%%%%%%%%%%%%
%%%%%%%%%%%%%%%%%%%%%%%%%%%%%%%%%%%%%%%%%%%%%%%%%%%%%%%%%%%
\section{An algorithm for the computation of the classes}\label{S4}

In this section we address the problem of finding an algorithm to compute the set $[S]$, for each  $S\in \mathscr I(m,F)$, where $[S]$ is the class of $S$ for the equivalence relation defined by $\theta(S)$ in Section \ref{S3}. In other words, in this section we will formulate the second algorithm mentioned in Remark \ref{remAlg}.

First, we recall that a numerical semigroup $S$ is said to be \emph{homogeneous} if it has no minimal generator in the interval $\Big[\frac{\operatorname{F}(S)}{2}, \operatorname{F}(S)\Big].$ We denote by $\mathscr H(m,F)$ the set of homogeneous numerical semigroups with multiplicity $m$ and Frobenius number $F$.

\begin{proposition}\label{10}
Let $m$ and $F$ be  positive integers such that $m\geq 3$, $F> 2m$,  $m\nmid F$. If $S\in \mathscr L(m,F)$, then $[S]\cap  \mathscr H(m,F)= \{Z([S])\}$. Moreover 
\[Z([S])=\left\langle\theta(S)\cup \{m\}\right\rangle\cup \{F+1,\rightarrow\}.\]
\end{proposition}

\begin{proof}
If $Z([S])$ does not belong to $\mathscr H(m,F)$, there exists a minimal generator $x$ of $Z([S])$ such that $\frac{F}{2}<x<F$. By using Lemma \ref{9}, we deduce that $Z([S])\backslash\{x\}\in [S]$ in contradiction with the minimality of $Z([S])$.

Next we see that if $S'\in [S]\cap  \mathscr H(m,F)$ then $S'=Z([S])$. Since $S'\in [S]$ it follows that $Z([S])\subseteq S'$. For the other inclusion, we consider $x\in S'$ and we distinguish three cases. 
\begin{itemize}
\item If $x<\frac{F}{2}$ then we have $x\in Z([S])$, since $S'  \sim \big( Z([S]) \big)$.
\item If $\frac{F}{2}<x < F$ then $x\in Z([S])$, because $S'$ has no minimal generators in the interval $\Big[\frac{F}{2}, F\Big]$.
\item If $x> F$ then $x\in Z([S])$ because $\operatorname{F}(Z([S]))=F$.
\end{itemize}
Finally, observe that $\left\langle\theta(S)\cup \{m\}\right\rangle\cup \{F+1,\rightarrow\}\in
\mathscr H(m,F)$ and thus  $\left\langle\theta(S)\cup \{m\}\right\rangle\cup \{F+1,\rightarrow\}=Z([S])$.
\end{proof}

As a consequence of Theorem \ref{6} and Proposition \ref{10}, we obtain the following result.

\begin{corollary}\label{11}
Let $m$ and $F$ be  positive integers such that $m\geq 3$, $F> 2m$ and $m\nmid F$. If $S\in \mathscr I(m,F)$, then  a numerical $S'$ belongs to $[S]$ if and only if $Z([S]) \subseteq S'\subseteq S$.
\end{corollary}

For a given numerical semigroup  $S\in \mathscr I(m,F)$, we write \[D(S) = S\backslash Z([S]).\] Given two subsets $A$ and $B$ of $\mathbb{N}$, we write $A+B$ for the set $\{a + b \mid a \in A,\ b \in B\}$.

\begin{lemma}\label{12}
Let $m$ and $F$ be  positive integers such that $m\geq 3$, $F> 2m$,  $m\nmid F$. If $S\in \mathscr I(m,F)$ and $B\subseteq D(S)$, then $$\overline{S}=Z([S])\cup  \Big(\big(B+Z([S])\big)\cap D(S)\Big)\in [S].$$ Moreover, all the elements in $[S]$ are in that form. 
\end{lemma}

\begin{proof}
We first observe that $Z([S])\subseteq \overline{S}\subseteq S$. So, if we prove that $\overline{S}$ is a numerical semigroup, then we will have that $\overline{S}\in [S]$ by Corollary \ref{11}. Clearly, it suffices to see that the sum of two elements of $B$ belongs to $\overline{S}$. Now, since $B\subseteq D(S)$, all the elements in $B$ are greater than $\frac{F}{2}$. Therefore the sum of two elements in $B$ belongs to  $Z([S]) \subseteq \overline{S}$ and we are done.

To prove the second statement, consider $S'\in [S]$. By Corollary \ref{11}, there exists $B \subseteq D(S)$ such that $S'=Z([S])\cup B$. As $S'$ is a numerical semigroup, we conclude that $S'= Z([S])\cup \left(\left(B+Z([S])\right)\cap D(S)\right)$.
\end{proof}

Given $S\in\mathscr I(m,F)$ and $B \subseteq D(S)$, we write \[T(B)=\bigcup_{b\in B} \left(\{b\}+ Z([S])\right) \cap D(S).\]

The next result is a reformulation of Lemma \ref{12} with this new notation.

\begin{proposition}\label{13}
Let $m$ and $F$ be  positive integers such that $m\geq 3$, $F> 2m$,  $m\nmid F$. If $S\in \mathscr I(m,F)$ and
$A=\{T(B) ~|~ B\subseteq D(S) \}$, then $[S] =\{Z([S])\cup X ~|~ X\in A\}$.
\end{proposition}

Observe that the lattice structure of $[S]$ is the same as the lattice structure of $A = \{T(B) ~|~ B\subseteq D(S) \}$. Now, we are ready to give an algorithmic procedure to compute the class $[S]$ of $S\in \mathscr I(m,F)$.

\begin{algorithm}\label{14}\mbox{}\par
{Input:} $S\in \mathscr I(m,F)$.\par
{Output:} $[S]$.
\begin{itemize}
\item[1.] Set $Z([S]) := \left\langle\theta(S)\cup \{m\}\right\rangle\cup \{F+1,\rightarrow\}$ and $D(S) := S\backslash Z([S])$.
\item[2.] Compute the set $A=\{T(B) ~|~ B\subseteq D(S) \}$.
\item[3.] Return the set $\{Z([S])\cup X ~|~ X\in A\}$.
\end{itemize}
\end{algorithm}

\begin{remark}
Notice that the above algorithm completes the computation of $\mathscr L(m,F)$ (see Remark \ref{remAlg}). Thus, we can use Algorithms \ref{24}
and \ref{14} to compute the whole set of numerical semigroups with fixed multiplicity and Frobenius number. This can be done with the function 
\begin{center}
\texttt{NumericalSemigroupsWithFrobeniusNumberAndMultiplicity} 
\end{center}
in\-cluded in the GAP \cite{gap} package \texttt{NumericalSgps} \cite{numericalsgps}.
\end{remark}

Let us see in an example how our algorithm works.

\begin{example}\label{16}
By Lemma \ref{15}, we have that $S=\left\langle 5,7,9,11\right\rangle\in \mathscr I(5,13)$.
Let us compute $[S]$ by using Algorithm \ref{14}. In this case, $\theta(S) = \varnothing$, therefore $Z([S])=\left\langle 5\right\rangle\cup\{14,\rightarrow\}$ and $D(S)=\{7,9,11,12\}.$ Since, $T(\{7\}) = \{7, 12\}, \ T(\{9\}) = \{9\},\ T(\{11\}) = \{11\}$ and $T(\{12\}) = \{12\}$, we obtain that 
\begin{align*}
A = \big\{ & \varnothing, \{9\}, \{11\}, \{12\}, \{7,12\}, \{9,11\}, \{9,12\}, \{11,12\},\\ & \{7,9,12\}, \{7,11,12\}, \{9,11,12\},  \{7,9,11,12\} \big\}.
\end{align*}
and thus 
$[S]=\{\left\langle 5\right\rangle\cup X\cup\{14,\rightarrow\} ~|~ X\in A\}.$
\end{example}

Observe that we can combine Algorithms \ref{14} and \ref{24} to compute simultaneously both the irreducible semigroups and the elements in the classes $[S]$, because we can compute the set $Z([S])$ by using the construction given in Theorem \ref{23}. More precisely:

\begin{corollary}
Let $m$ and $F$ be  positive integers such that $F\geq 3$,  $m\leq \frac{F+2}{2}$ and $m\nmid F$. If $T$ and $S \in\mathscr I(m,F)$ with $\psi(S) = T$, then $$Z([S]) = \langle \theta(T) \cup \{m, F-x\} \rangle \cup \{F+1, \rightarrow \},$$ for some $x\in\operatorname{msg}(T)$ such that  $\frac{F}{2}<x<F$,  $2x-F\not\in S$,  $3x\neq 2F$, $4x\neq 3F$ and $\operatorname{m}(T)<F-x< \operatorname{r}(T)$.
\end{corollary}

\begin{proof}
By Theorem \ref{23}, $T = \big(S\backslash \{x\}\big)\cup  \big\{F-x\big\}$ for some $x\in\operatorname{msg}(S)$ such that  $\frac{F}{2}<x<F$,  $2x-F\not\in S$,  $3x\neq 2F$, $4x\neq 3F$ and $\operatorname{m}(S)<F-x< \operatorname{r}(S)$. Then $\langle \theta(S) \cup \{m\} \rangle = \langle \theta(T) \cup \{m, F-x\} \rangle$ and, by Proposition \ref{10}, we are done.
\end{proof}

This can open a door to potentially faster implementations of our algorithm for computing $\mathscr L(m,F)$.

%%%%%%%%%%%%%%%%%%%%%%%%%%%%%%%%%%%%%%%%%%%%%%%%%%%%%%%%%%%
\subsection*{On the (binomial ideal) structure of $[S]$.} We end this section by noticing that $[S],\ S \in \mathscr I(m,f)$, can be described in terms of certain binomial ideals in a polynomial ring over a field.

Recall that by Lemma \ref{4}, $[S]$ is a semigroup of sets with respect to the union. Moreover, by Proposition \ref{13}, $[S]$ is generated by $\{T(\{d\}) \mid d \in D(S)\}$. Therefore, if $D(S) = \{d_1, \ldots, d_n\}$, we have the following semigroup homomorphism $$\varphi : \mathbb{N}^n \to [S]; \mathbf{e}_i \mapsto Z([S]) \cup T(\{d_i\}),\ i = 1, \ldots, n,$$ by convention $\varphi(0) = \varnothing.$ Associated to this homomorphism, we have the following binomial ideal $$I_{[S]} = \langle \mathbf X^\mathbf{u} - \mathbf X^\mathbf{v}\ \mid\ \varphi(\mathbf{u}) = \varphi(\mathbf{v}) \rangle \subseteq \Bbbk[X_1, \ldots, X_n],$$ where $\mathbf X^\mathbf{u} = X_1^{u_1} \cdots X_n^{u_n},\ \mathbf{u} = (u_1, \ldots, u_n) \in \mathbb{N}^n$. 

On can prove that the dimension of $\Bbbk[X_1, \ldots, X_n]/I_{[S]}$ as $\Bbbk-$vector space is equal to the cardinality of $[S]$, which coincides with the cardinality of $A = \{T(B) \mid B \subseteq D(S)\}$ by Proposition \ref{13}. In fact, $I_{[S]}$ is the ideal of a finite set of points $\mathcal Z$ with coordinates in $\{0,1\}$ and there is a natural one-to-one correspondence \[\mathcal Z \to [S]; (x_1, \ldots, x_n) \mapsto Z([S])\cup \{ d_i \mid x_i = 1\}.\] Therefore, it would be possible to compute the set $[S]$ from the primary decomposition of $I_{[S]}$ and vice versa.

\begin{example}
The binomial ideal, $I_{[S]} \subset \Bbbk[X_1, X_2, X_3, X_4]$, associated to $[S]$ in Example \ref{16} is the ideal generated by $\{ X_1^2 - X_1, \ldots, X_4^2 - X_4, X_1 X_4 - X_1\}.$ The set of zeros of $I_{[S]}$ consists of $12$ points in $\Bbbk^4$, more precisely \[I_{[S]} = \sqrt{I_{[S]}} = \bigcap_{i=1}^{12} \mathfrak{m}_i,\] with
\[\mathfrak{m}_1 = (X_1,X_2,X_3,X_4),\ \mathfrak{m}_2 = (X_1,X_2-1,X_3,X_4),\]
\[\mathfrak{m}_3 = (X_1,X_2,X_3-1,X_4),\ \mathfrak{m}_4 = (X_1,X_2,X_3,X_4-1), \]
\[\mathfrak{m}_5 = (X_1-1,X_2,X_3,X_4-1),\ \mathfrak{m}_6 = (X_1,X_2-1,X_3-1,X_4), \]
\[\mathfrak{m}_7 = (X_1,X_2-1,X_3,X_4-1),\ \mathfrak{m}_8 = (X_1,X_2,X_3-1,X_4-1), \]
\[\mathfrak{m}_9 = (X_1-1,X_2-1,X_3,X_4-1),\ \mathfrak{m}_{10} = (X_1-1,X_2,X_3-1,X_4-1)\]
\[ \mathfrak{m}_{11} = (X_1,X_2-1,X_3-1,X_4-1)\] and \[\mathfrak{m}_{12} = (X_1-1,X_2-1,X_3-1,X_4-1).\] With this notation, we have for example that the ideal $\mathfrak{m}_7$ corresponds to the numerical semigroup $Z([S])\cup \{9,11\} = \langle 5 \rangle\cup \{9,11\} \cup\{14,\rightarrow\}$ or that the ideal $\mathfrak{m}_{10}$ correspond to  $Z([S])\cup \{7,11,12\} = \langle 5 \rangle\cup \{7,11,12\} \cup\{14,\rightarrow\}$.
\end{example}

%%%%%%%%%%%%%%%%%%%%%%%%%%%%%%%%%%%%%%%%%%%%%%%%%%%%%%%%%%%
%%%%%%%%%%%%%%%%%%%%%%%%%%%%%%%%%%%%%%%%%%%%%%%%%%%%%%%%%%%
\section{The set of numerical semigroups of a given multiplicity and genus}\label{S7}

Following the same idea as that expressed Remark \ref{remAlg}, we can formulate an analogous type of algorithm to compute the set of all numerical semigroups with multiplicity $ m $ and genus $ g $ that we denote by $\widehat{\mathscr{L}}(m, g) $. 

The next result characterizes the integers $m$ and $g$ such that there exists a numerical semigroup with multiplicity $m$ and genus $g$.

\begin{lemma}\label{26} \cite[Proposition 2.1]{IJNT}.
If $\left(m,g\right)\in \mathbb{N}\backslash \{0\}\times \mathbb{N}$, then  $\widehat{\mathscr {L}}(m,g)\neq\varnothing$ if and only if $\left(m,g\right)=\left(1,0\right)$ or 
$2\leq m \leq g+1$.
\end{lemma}

Since  $\widehat{\mathscr {L}} (g+1, g)=\{0,g+1\rightarrow \}$, from now on we will assume $2\leq m \leq g$.

The following result determines a necessary and sufficient condition for an integer $F$ to be the Frobenius number of a numerical semigroup of multiplicity $m$ and genus $g$.

\begin{proposition}\label{27} \cite[Theorem 2.4]{IJNT}.
Let $m$, $g$  and $F$ be  positive integers. If $2\leq m \leq g$ and  $m\nmid F$, then there exists $S\in \widehat{\mathscr {L}} (m, g)$ with $\operatorname{F}(S)=F$ if and only if $\lceil \frac{m g}{m-1} \rceil-1\leq F\leq 2g-1$.
\end{proposition}

We denote by $\widehat{\mathscr {L}} (m, g, F)$ the set of all numerical semigroups with  multiplicity $m$, genus $g$ and Frobenius number $F$. As a consequence of Proposition \ref{27} we have the following result.

\begin{corollary}\label{28}
Let $m$, $g$  and $F$ be  positive integers. If $2\leq m \leq g$ and $\mathcal{B}_{m,g}:=\big\{F\in \{\lceil \frac{m g}{m-1} \rceil-1,\ldots, 2g-1\} ~|~  m\nmid F \big\}$, then $$\widehat{\mathscr {L}}(m,g)=\bigcup_{F\in \mathcal{B}_{m,g}}\widehat{\mathscr {L}} (m, g, F).$$
\end{corollary}
 
Therefore, to compute all elements in  $\widehat{\mathscr {L}}(m, g)$, it is enough to compute the elements in $\widehat{\mathscr {L}} (m, g, F)$ for each $F\in\mathcal{B}_{m,g}$.

The next algorithm is a reformulation of Algorithm \ref{14}, for the computation of the set of elements in the class of $S \in \mathscr I(m,F)$ with respect to $\sim$ with genus $g$.

\begin{algorithm}\label{29}\mbox{}\par
{Input:} $S\in \mathscr I(m,F)$ and $g$.\par
{Output:} $\{T\in [S] ~|~ \operatorname{g}(T)=g \}$
\begin{itemize}
\item[1.] Set $Z([S]) := \left\langle\theta(S)\cup \{m\}\right\rangle\cup \{F+1,\rightarrow\}$ and $D(S) := S\backslash Z([S])$.
\item[2.] Set $A=\{ T(B) ~|~ B\subseteq D(S) ~\textrm{and} ~\# T(B):=g\left(Z([S])-g\right)\}$
\item[3.] Return the set $\{Z([S])\cup X ~|~ X\in A\}$.
\end{itemize}
\end{algorithm}

We illustrate the above algorithm with the following example.

\begin{example}\label{30}
From Example \ref{16} we know that $S=\langle 5,7,9,11\rangle\in \mathscr I (5, 13)$. By using Algorithm \ref{29} we compute the set 
$\{T\in [S] ~|~ \operatorname{g}(T)=10 \}$. Since $Z([S])=\left\langle 5 \right\rangle\cup \{14,\rightarrow\},$ $D(S)=\{7,9,11,12\}$ and $\operatorname{g}\left(Z([S])\right)=11$, we have that \[A=\{ T(B) ~|~ B\subseteq D(S) ~\textrm{and} ~\# T(B)=11-10=1 \} =  \{\{9\},\{11\},\{12\}\}.\] Hence,  $\{T\in [S] ~|~ \operatorname{g}(T)=10 \}=\{\langle 5\rangle\cup X\cup\{14\rightarrow\} ~|~X\in A\}=$
$\{\langle 5,9,16,17\rangle,$ $\langle 5,11,14,17,18\rangle,$ $\langle 5,12,14,16,18\rangle \}$.
\end{example}

Observe that by just controlling the cardinality of $T(B)$ suitably, we can modify our GAP \cite{gap} code in Section \ref{S4} to produce the corresponding function which computes $\widehat{\mathscr {L}} (m, g, F)$.

\begin{remark}
For $q > 1$ and $g \geq 2$ we set \[\overline{\mathcal{B}}_{m,g} := \Big\{F\in \big\{\Big\lceil \frac{m g}{m-1} \Big\rceil-1,\ldots 2g-1 \big\} \cap \big\{(q-1)m + 1, \ldots, qm\big\} ~|~  m\nmid F \Big\}\] for each $m = 2, \ldots, g$. Clearly, 
\[
\bigcup_{m=2}^g \left(\bigcup_{F\in \overline{\mathcal{B}}_{m,g}}\widehat{\mathscr {L}} (m, g, F) \right)
\]
is the set of numerical semigroups with genus $g$ and depth $q$ (see Remark \ref{Depth1}). If $q=2$, this formula is rather explicit by Proposition \ref{3}. For $q > 2$, we can take advantage of our algorithms to compute this set.
\end{remark}

\noindent \textbf{Acknowledgments.} We thank the referees for their careful review and helpful comments. We also thank Pedro A. Garc\'{\i}a-S\'anchez and Manuel Delgado for their help during the implementation of the algorithms in GAP \cite{gap}. 

%%%%%%%%%%%%%%%%%%%%%%%%%%%%%%%%%%%%%%%%%%%%%%%%%%%%%%%%%%%
%%%%%%%%%%%%%%%%%%%%%%%%%%%%%%%%%%%%%%%%%%%%%%%%%%%%%%%%%%%

\end{document}